\documentclass[a4paper,11pt]{article}
\usepackage[top=2.5cm,bottom=2.5cm,left=2.2cm,right=2.2cm]{geometry}
\usepackage{amsfonts}
\usepackage{mathrsfs,amscd,amssymb,amsthm,amsmath,bm,graphicx,psfrag,subfigure,url,mathtools}
\usepackage{pict2e}
\usepackage{stfloats}
\usepackage{psfrag,amsmath}
\usepackage{tikz}
\usepackage{indentfirst}
\usepackage{hyperref}
\usepackage{bookmark}
\usepackage{enumerate}
\usepackage{latexsym,euscript,epic,eepic,color}
\usepackage{multirow}
\usepackage{multicol}
\usepackage{longtable}
\usepackage{adjustbox}
\usepackage[all]{xy}
\usepackage{setspace}
\usepackage{epstopdf}
\allowdisplaybreaks
\usepackage{authblk}

\makeatletter

\renewcommand{\@seccntformat}[1]{{\csname the#1\endcsname}{\normalsize .}\hspace{.5em}}
\makeatother

\usepackage{ifpdf}
\usepackage{indentfirst}

\def \[{\begin{equation}}
\def \]{\end{equation}}

\def \diam{{\rm diam}}

\newtheorem{thm}{Theorem}[section]
\newtheorem{prop}[thm]{Proposition}
\newtheorem{defi}{Definition}
\newtheorem{claim}{Claim}
\newtheorem{lem}[thm]{Lemma}
\newtheorem{cor}[thm]{Corollary}

\newtheorem{remark}{Remark}

\newenvironment{wst}
{\setlength{\leftmargini}{1.5\parindent}
 \begin{itemize}
 \setlength{\itemsep}{-1.1mm}}
{\end{itemize}}

\begin{document}

\setlength{\baselineskip}{15pt}
\begin{center}{\Large \bf The Hoffman program for mixed graphs}
\vspace{4mm}

{\large Yuantian Yu$^{a}$,\, Edwin R. van Dam$^b$\footnote{Corresponding author. This work is supported by the CSC scholarship program (202306770059). \\
\hspace*{5mm}{\it Email addresses}: ytyumath@sina.com (Y.T. Yu), Edwin.vanDam@tilburguniversity.edu (E.R. van Dam).}}\vspace{2mm}

$^a$School of Science, East China University of Technology, Nanchang 330013, China

$^b$Department of Econometrics and O.R., Tilburg University, Tilburg, Netherlands
\end{center}

\noindent {\bf Abstract}:\, We consider Hoffman's program about the limit points of the spectral radius of the Hermitian adjacency matrix of mixed graphs. {In particular, we determine all such limit points. As an intermediate step, we determine all mixed graphs without negative $4$-cycle whose spectral radius does not exceed $\sqrt{2+\sqrt{5}}$.}

\vspace{2mm} \noindent{\it Keywords:}
Mixed graph; Spectral radius; Limit point; $\sqrt{2+\sqrt{5}}$

\vspace{2mm}

\section{\normalsize Background}\setcounter{equation}{0}
Let $\mathcal{G}$ be a class of graphs, and let $Q(G)$ be a square matrix associated to $G\in \mathcal{G}$, with $Q$-\textit{spectral radius} $\rho(Q(G)).$ A real number $\gamma$ is said to be a \textit{limit point} of $Q$-spectral radii ($Q$-\textit{limit point} for short) of items in $\mathcal{G}$ if there exists a sequence $(G_k)_{k\in \mathbb{N}}$ in $\mathcal{G}$ such that
$$
\text{$\rho(Q(G_i))\neq \rho(Q(G_j))$ whenever $i\neq j,$ and $\lim_{k\rightarrow\infty}\rho(Q(G_k))=\gamma.$}
$$
The two facets of the \textit{Hoffman program} with respect to the matrix $Q$ {for the class} $\mathcal{G}$ are posed by Hoffman \cite{H1972} ({see} also \cite{BB2024}):
\begin{wst}
\item[{\rm $(\ast)$}] establishing all the possible $Q$-limit points;
\item[{\rm $(\ast\ast)$}] characterizing the graphs in $\mathcal{G}$ whose $Q$-spectral radius does not exceed a fixed limit point.
\end{wst}

Let $G =(V(G), E(G))$ be a simple graph with $(0, 1)$-\textit{adjacency matrix} $A(G).$ When $\mathcal{G}$ is a class consisting of all simple graphs, Hoffman \cite{H1972} determined all the $A$-limit points of items in $\mathcal{G}$ up to {$\rho^\ast:=\sqrt{2+\sqrt{5}}$;} Shearer \cite{S1989} proved $\gamma$ is an $A$-limit point of items in $\mathcal{G}$ for all {$\gamma\geq \rho^\ast$,} which solved the first part of the Hoffman program with respect to the adjacency matrix. Smith \cite{S1970} characterized all simple graphs whose $A$-spectral radius does not exceed $2$ (the smallest $A$-limit point); Brouwer, Neumaier \cite{BN1989} and Cvetkovi\'c, Doob, Gutman \cite{CDG1989} determined all simple graphs with $A$-spectral radius between $2$ and $\rho^\ast.$

Belardo and Brunetti \cite{BB2024} completely solved the first part of the Hoffman program with respect to the $(0, \pm1)$-adjacency matrix of signed graphs. For the second part, they showed that $2$ is also the smallest $A$-limit point for signed graphs. McKee and Smyth \cite{MS2007} identified all signed graphs whose $A$-spectral radius does not exceed $2;$ recently, Wang et al. \cite{WDHL2023} determined all signed graphs with $A$-spectral radius between $2$ and $\rho^\ast,$ which solved an open problem posed by Belardo et al. \cite{BCKW2018}.

A \textit{mixed graph} $M$ is obtained from a simple graph $G$ by orienting some of its edges; we call $G$ the underlying graph of $M,$ denoted by $G:=\Gamma(M).$ 
Guo and Mohar \cite{G2017} determined all mixed graphs whose $H$-spectral radius {(the $H$-matrix of a mixed graph is defined in Definition \ref{defi1} below)} is less than $2$, from which we know that $2$ is also the smallest $H$-limit point for mixed graphs. Yuan et al. \cite{YWGQ2020} determined all mixed graphs containing no mixed $4$-cycle whose $H$-spectral radius is at most $2.$ Greaves \cite{GG2012} characterized all the gain graphs with gains from the Gaussian or Eisenstein integers whose adjacency eigenvalues are contained in $[-2, 2],$ but this does not determine all mixed graphs whose $H$-spectral radius is at most $2$ (see \cite{GM2023}). Based on the result in \cite{GG2012}, Gavrilyuk and Munemasa \cite{GM2023} completed the characterization of mixed graphs whose $H$-spectral radius is at most $2.$
\begin{figure}
\begin{center}
\includegraphics[width=100mm]{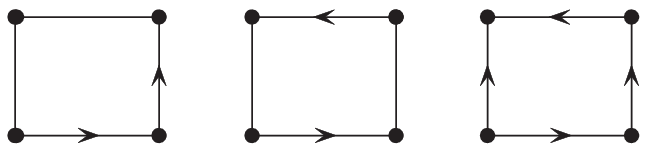} \\
  \caption{{Negative $4$-cycles.}}\label{fig00}
\end{center}
\end{figure}

Motivated by the above, we study the Hoffman program with respect to the $H$-matrix of mixed graphs. {We say that a mixed graph has no \textit{negative $4$-cycle} if it has none of the mixed graphs in Fig. \ref{fig00} as a subgraph (not necessarily induced).} In this paper, on one hand, we determine all mixed graphs without negative $4$-cycle whose $H$-spectral radius is at most $\rho^\ast.$ This solves the second part of the Hoffman program with respect to the $H$-matrix of $\mathcal{G}$ and the limit value {$\rho^\ast,$} where $\mathcal{G}$ is the class consisting of all mixed graphs without negative $4$-cycle. On the other hand, based on this result and the results in \cite{H1972,S1989,BB2024}, we identify all $H$-limit points of mixed graphs, which completely solves the first part of the Hoffman program.

The remainder of this paper is organized as follows: In Section 2, we give some preliminary results, which will be used 
in the subsequent sections. In Section 3, we determine all mixed graphs without negative $4$-cycle whose $H$-spectral radius is at most {$\rho^\ast.$} In Section 4, we identify all $H$-limit points of mixed graphs.

\section{\normalsize Definitions and preliminaries}\setcounter{equation}{0}

As usual, let $P_n, C_n$ and $K_n$ denote the path, cycle and complete graph on $n$ vertices, respectively; and let $K_{m,n}$ denote the complete bipartite graph with the orders of partite sets being $m$ and $n.$ The diameter of a simple graph $G$ is denoted by $\diam(G).$

For a mixed graph $M$, the numbers of vertices $|V(M)|$ and edges $|E(M)|$ in $M$ are called the \textit{order} and \textit{size} of $M$, respectively. We say that two vertices $u$ and $v$ are \textit{adjacent} (or \textit{neighbours}) if there is an arc or an undirected edge between them and we write it as $u\sim v$. We write an undirected edge as $\{u,v\}$ and a directed edge (or an arc) from $u$ to $v$ as $\overrightarrow{uv}.$ Usually, we denote an edge of $M$ by $uv$ if we are not concerned whether it is directed or not. For a mixed graph $M$ and a vertex $v$, define $d_M(v)$ as the \textit{degree} of $v$ in $M$, i.e., the number of vertices adjacent to $v.$ The largest degree of $M$ is denoted by $\Delta(M).$

A mixed graph $M'$ is an (induced) subgraph of $M$, denoted by $M'=M[\Gamma(M')]$, if $\Gamma(M')$ is an (induced) subgraph of $\Gamma(M)$ and the direction of each arc in $M'$ coincides with that in $M$. For a vertex subset $V'$ of $V(M),$ $M[V']$ is the subgraph of $M$ induced on $V'$. By $M-u$, we denote the mixed graph obtained from $M$ by deleting the vertex $u \in V(M)$. A mixed graph is called \textit{connected} if its underlying graph is connected. 
The \textit{girth} of a mixed graph $M$, denoted by $g(M),$ is the length of the shortest cycle contained in $\Gamma(M).$ 


The \textit{Hermitian adjacency matrix} ($H$-\textit{matrix} for short) was proposed by Guo, Mohar \cite{GM2017} and Liu, Li \cite{LL2015}, independently.
\begin{defi}[\cite{GM2017,LL2015}]\label{defi1}
Let $M$ be a mixed graph, then the $H$-matrix $H(M)=(h_{uv})$ for $M$ is defined as 
$$
h_{uv}=\left\{
  \begin{array}{rl}
    {\bf i} & \text{if $\overrightarrow{uv}$ is an arc for $M$;} \\
    -{\bf i} & \text{if $\overrightarrow{vu}$ is an arc for $M$;} \\
    1 & \text{if $\{u,v\}$ is an undirected edge for $M$;} \\
    0 & \text{otherwise},
  \end{array}
\right.
$$
where ${\bf i}$ is the imaginary unit.
\end{defi}
Clearly, $H(M)$ is Hermitian, hence its eigenvalues (also called the \textit{eigenvalues} of $M$) are all real. We denote the eigenvalues of $M$ by
$\lambda_1(M) \geq \lambda_2(M) \geq \cdots  \geq \lambda_{|V(M)|}(M).$
The collection of eigenvalues of $M$ (with multiplicities) is called the \textit{spectrum} of $M$. Two mixed graphs are called \textit{cospectral} if they have the same spectrum. The \textit{spectral radius} of $M$, written as $\rho(M)$, is defined to be the largest absolute value of its eigenvalues.

Given a mixed graph $M$, let $M^T$ be its \textit{converse} (the mixed graph obtained by reversing all the arcs of $M$). It is immediate from Definition~\ref{defi1} that $H(M^T)$ is the transpose of $H(M)$. This implies that $M$ and $M^T$ are cospectral.

Guo and Mohar \cite{GM2017} proposed \textit{four-way switching} on mixed graphs, which can be summarized as follows (see also \cite{PW2020}).
\begin{defi}[\cite{GM2017}]\label{defi2}
Let $M$ be a mixed graph. A four-way switching is the operation of changing $M$ into the mixed graph $M'$ by choosing an appropriate diagonal matrix $D$ with $D_{ii}\in\{\pm1, \pm{\bf i}\}, \,i=1,\ldots, |V(M)|,$ and $H(M)=D^{-1}H(M')D.$
\end{defi}

By appending the four-way switching with vertex relabeling and taking the converse, the following is a natural notion of equivalence, which is motivated from Wissing and van Dam~\cite{Pv2022}.
\begin{defi}\label{defi3}
Two mixed graphs are said to be \textit{switching isomorphic} if one may be obtained from the other by a sequence of four-way switchings and
vertex permutations, possibly followed by taking the converse.
\end{defi}

Clearly, two mixed graphs are cospectral if they are switching isomorphic.

A \textit{mixed cycle} is a mixed graph whose underlying graph is a cycle. A mixed cycle is \textit{even} (resp. \textit{odd}) if its order is even (resp. odd). Let $\widetilde{C}$ be a mixed cycle with $\Gamma(\widetilde{C})=v_1v_2v_3\cdots v_{l-1}v_lv_1$. Then the weight of $\widetilde{C}$ in a direction is defined by
$$
wt(\widetilde{C})=h_{12}h_{23}\cdots h_{(l-1)l}h_{l1},
$$
where $h_{jk}$ is the $(v_j,v_k)$-entry of $H(\widetilde{C})$. Note that $h_{jk}\in \{1,\pm{\bf i}\}$ if $v_j\sim v_k$, so we have $wt(\widetilde{C})\in \{\pm1, \pm{\bf i}\}.$ Note that if, for one direction, the weight of a mixed cycle is $\alpha$, then for the reversed direction its weight is $\bar{\alpha}$, the conjugate of $\alpha$. For a mixed cycle $\widetilde{C}$, it is \textit{positive} (resp. \textit{negative}) if $wt(\widetilde{C})=1$ (resp. $-1$); it is \textit{imaginary} if $wt(\widetilde{C})\in \{\pm{\bf i}\}$. Further, we call a mixed cycle \textit{real} if it is positive or negative.


\begin{lem}[\cite{LL2015}]\label{lem2.02}
If $M$ is a mixed graph of order $n$ without real mixed odd cycles, then its spectrum is symmetric about zero.
\end{lem}

Now we give an equivalent condition for switching isomorphism, which may be used to determine quickly whether or not a given pair of mixed graphs is switching isomorphic.
\begin{lem}[\cite{WY2020}]\label{lem2.4}
Let $M$ and $M'$ be two connected mixed graphs with the same underlying graph $G.$ Then $M$ and $M'$ are switching isomorphic if and only if there is a mixed graph $M''$ isomorphic to $M$ with $wt(M''[C])=wt(M'[C])$ for every cycle $C$ in $G.$
\end{lem}

By Lemma~\ref{lem2.4}, the following lemma is clear, see also \cite{G2017}.
\begin{lem}[\cite{G2017}]\label{lem2.5}
Let $F$ be a forest. Then every mixed graph whose underlying graph is isomorphic to $F$ is switching isomorphic to $F.$
\end{lem}

A \textit{signed graph} $\dot{G}=(G, \sigma)$ is a pair $(G, \sigma),$ where $G$ is a simple graph (called the underlying graph of $\dot{G}$), and $\sigma: E(G)\rightarrow\{+1, -1\}$ is a sign function. The adjacency matrix of $\dot{G}$ is defined as a $|V(G)|\times |V(G)|$ matrix $A(\dot{G}) = (a_{ij})$ with $a_{ij} = \sigma(v_iv_j)$ if $v_iv_j\in E(G),$ and $a_{ij} =0$ otherwise. 
For a signed graph $\dot{G}=(G,\sigma)$, the sign of a signed cycle $\dot{C}$ in $\dot{G}$ is $\sigma(\dot{C}) =\prod_{e\in \dot{C}} \sigma(e).$ The signed cycle $\dot{C}$ is \textit{positive} (resp. \textit{negative}) if $\sigma(\dot{C})=1$ (resp. $-1$).

Let $M$ be a mixed graph with $\Gamma(M)=G.$ We say $M$ is \textit{cycle-isomorphic} to a signed graph $\dot{G}=(G,\sigma)$, if there is a mixed graph $M'$ isomorphic to $M$ with $wt(M'[C])=\sigma(\dot{C})$ for every cycle $C$ in $G.$ 
By comparing the coefficients of the characteristic polynomial of the adjacency matrix for a signed graph (see \cite[Theorem 2.1]{BCKW2018}) and those of the characteristic polynomial of the Hermitian adjacency matrix for a mixed graph (see \cite[Theorem 2.8]{LL2015}), it follows that if a mixed graph $M$ is cycle-isomorphic to a signed graph $\dot{G}$, then $H(M)$ and $A(\dot{G})$ have the same characteristic polynomial, and so they are cospectral.
\begin{remark}
We note that by \cite[Theorem 2.2]{SK2019} it can be shown also that if a mixed graph $M$ is cycle-isomorphic to a signed graph $\dot{G},$ then it is also switching isomorphic to $\dot{G}$ in the sense that there is a diagonal $(\pm1,\pm {\bf i})$-matrix $D$ such that $H(M')=D^{-1}A(\dot{G})D,$ where $M'$ is isomorphic to $M.$
\end{remark}

Let $G$ be a connected graph with a fixed spanning tree $T$. Then for each edge $e\in E(G)\backslash E(T),$
adding the edge $e$ to $T$ creates a unique cycle in $T\cup\{e\}.$ This cycle is called a \textit{fundamental cycle} of $G$ associated with $T.$ {Let $M'$ (resp. $\dot{G}$) be a mixed graph (resp. signed graph) with underlying graph $G.$ For a cycle $C$ in $G,$ we denote $\dot{C}$ the signed cycle in $\dot{G}$ whose underlying graph is $C$.} A result of Samanta and Kannan \cite[Theorem 3.1]{SK2019} shows $wt(M'[C])=\sigma(\dot{C})$ for every cycle $C$ in $G$ if and only if $wt(M'[C])=\sigma(\dot{C})$ for every fundamental cycle $C$ of $G$ associated with $T.$
\begin{lem}\label{lem2.06}
Let $M$ be a connected mixed graph with $\Gamma(M)=G.$ If $M$ contains no imaginary mixed cycle, then $M$ is cycle-isomorphic to a signed graph $\dot{G}$.
\end{lem}
\begin{proof}
If $G$ is a tree, then the result is clear. Now let $G$ be a connected graph with at least one cycle. Take a spanning tree $T$ of $G,$ let $\{C^1,\ldots,C^t\}$ be the family of all fundamental cycles of $G$ associated with $T,$ and let $e_i$ be the edge in $E(C^i)\backslash E(T)$ for $i=1,\ldots,t.$ Since $M$ contains no imaginary mixed cycle, $wt(M[C^i])\in \{\pm 1\}.$ Construct a sign function $\sigma$ on $E(G)$ as follows: $\sigma(e)=1$ for all $e\in E(T)$ and $\sigma(e_i)=wt(M[C^i])$ for all $i=1,\ldots,t$. Then the signed graph $\dot{G}=(G,\sigma)$ has the property that $\sigma(\dot{C^i})=wt(M[C^i])$ for all $i=1,\ldots,t$, and so $\sigma(\dot{C})=wt(M[C])$ for every cycle $C$ in $G.$ 
\end{proof}
Denote the characteristic polynomial of an $n$-vertex ($n\geq1$) mixed graph $M$ by\footnote{$\Phi(M,x)=1$ for a mixed graph $M$ on $0$ vertices.}: $\Phi(M,x)=\det(xI-H(M)).$ 
Yuan et al. \cite[Corollary~2.6]{YWGQ2020} presented a recurrence relation for the characteristic polynomials of mixed graphs.
\begin{lem}[\cite{YWGQ2020}]\label{lem2.3}
Let $M$ be a mixed graph, and let $u$ be a vertex of $M$. Then
$$
\Phi(M,x)=x\Phi(M-u,x)-\sum_{v\sim u}\Phi(M-u-v,x)-2\sum_{\widetilde{C}\in \mathscr{C}^+(u)}\Phi(M-V(\widetilde{C}),x)+2\sum_{\widetilde{C}\in \mathscr{C}^-(u)}\Phi(M-V(\widetilde{C}),x),
$$
where $\mathscr{C}^+(u)$ (resp. $\mathscr{C}^-(u)$) is the set of positive (resp. negative) mixed cycles containing $u$.
\end{lem}
The following is a direct corollary of Lemma~\ref{lem2.3}.
\begin{cor}\label{cor2.03}
Let $\widetilde{C}$ be an imaginary mixed cycle of length $k,$ then
$$
\Phi(\widetilde{C},x)=x\Phi(P_{k-1},x)-2\Phi(P_{k-2},x).
$$
\end{cor} 
\begin{lem}[\cite{S1974,WBH2010}]\label{lem2.6}
The characteristic polynomial of the path satisfies the recurrence relation
\[\label{eq:2.2}
\text{$\Phi(P_0,x)=1,\,\,\,\Phi(P_1,x)=x,\,\,\,\Phi(P_n,x)=x\Phi(P_{n-1},x)-\Phi(P_{n-2},x)$ for $n\geq2$}.
\]
Moreover, the recurrence relation gives
\[\label{eq:2.3}
\text{$\Phi(P_n,x)=\frac{\varphi(x)^{2n+2}-1}{\varphi(x)^{n+2}-\varphi(x)^{n}}$, where $\varphi(x):=\frac{x+\sqrt{x^2-4}}{2}.$}
\]
\end{lem}
\begin{lem}\label{lem2.04}
If $M$ is a mixed graph, then $\sqrt{\Delta(M)}\leq\rho(M)\leq \Delta(M).$
\end{lem}
\begin{proof}
The upper bound for $\rho(M)$ follows by \cite[Theorem 5.1]{GM2017}. To obtain the lower bound, let $V(M)=\{v_1,\ldots,v_n\}$ with $d_M(v_1)=\Delta(M)$, and let ${\bf e}$ be {the} vector of length $n$ with ${\bf e}_1=1$ and ${\bf e}_i=0$ for $2\leq i\leq n.$ Then ${\bf e}^TH^2(M){\bf e}=d_M(v_1)=\Delta(M).$ On the other hand, by the min–max formula (see \cite[Equation (1)]{GM2017}), $\rho^2(M)\geq\lambda_1^2(M)\geq {\bf e}^TH^2(M){\bf e}.$ Thus $\rho(M)\geq \sqrt{\Delta(M)}.$  
\end{proof}

Assume that $\lambda_1\geq\lambda_2\geq\cdots\geq\lambda_n$ and $\mu_1\geq\mu_2\geq\cdots\geq\mu_{n-t}$ (where $t\geq1$) are two sequences of real numbers. We say that the sequence $\mu_i\, (1\leq i\leq n-t)$ \textit{interlaces} the sequence $\lambda_j\, (1\leq j\leq n)$ if for every $s=1,\ldots,n-t$, we have
$
\lambda_s\geq\mu_s\geq\lambda_{s+t}.
$
\begin{lem}[\cite{GM2017}]\label{lem2.1}
Let $M$ be a mixed graph, and let $M'$ be an induced subgraph of $M.$ Then the eigenvalues of $M'$ interlace the eigenvalues of $M$.
\end{lem}
\begin{lem}\label{lem2.05}
Let $(M_k)_{k\in \mathbb{N}}$ be a sequence of connected mixed graphs such that 
$$
\text{$\rho(M_i)\neq \rho(M_j)$ whenever $i\neq j,$ and $\lim_{k\rightarrow\infty}\rho(M_k)=\gamma$}
$$
for some constant $\gamma$. Then $\lim_{k\rightarrow\infty}\diam(\Gamma(M_k))=+\infty.$
\end{lem}
\begin{proof}
By Lemma \ref{lem2.04}, $\sqrt{\Delta(M_k)}\leq\rho(M_k)$ for each $k\in \mathbb{N}.$ By this inequality and the convergence of $(\rho(M_k))_{k\in \mathbb{N}}$, we know that there is a constant $c$ such that $\sqrt{\Delta(M_k)}\leq c$ for each $k\in \mathbb{N},$ and so $\Delta(\Gamma(M_k))=\Delta(M_k)\leq c^2.$ Since $M_k\,(k\in \mathbb{N})$ are pairwise distinct, $\lim_{k\rightarrow\infty}|V(\Gamma(M_k))|=+\infty.$ Now by the well-known inequality
$$
|V(\Gamma(M_k))|\leq \Delta(\Gamma(M_k))^{\diam(\Gamma(M_k))}+1\leq c^{2\diam(\Gamma(M_k))}+1,
$$
one has $\lim_{k\rightarrow\infty}\diam(\Gamma(M_k))=+\infty.$ 
\end{proof}

{Finally, we extend Hoffman's classical tool to compute the limit of the spectral radius of a sequence of graphs obtained from a graph $G$ by attaching a path of strictly increasing length to a fixed $v\in V (G),$ see \cite[Lemma 3.4]{H1972}. By the same discussion as the proof of \cite[Lemma 3.4]{H1972}, one may obtain the following result.
\begin{lem}\label{lem2.07}
Let $v$ be a vertex of a mixed graph $M,$ and let $M_n$ be the mixed graph obtained by attaching the path $P_{n+1}$ to $v$. If $\lim_{n\rightarrow\infty}\rho (M_n)=\lim_{n\rightarrow\infty}\lambda_1(M_n)>2,$ then $\lim_{n\rightarrow\infty}\rho (M_n)$ is the largest positive root of the function
$$
\text{$\varphi(x)\Phi(M, x)-\Phi(M-v, x),$ where $\varphi(x)=\frac{x+\sqrt{x^2-4}}{2}.$}
$$
\end{lem}
}

\section{\normalsize Mixed graphs without negative $4$-cycle whose spectral radius is at most $\sqrt{2+\sqrt{5}}$}\setcounter{equation}{0}
In this section, we are to characterize all mixed graphs (up to switching isomorphism) without negative $4$-cycle whose spectral radius is at most $\rho^\ast:=\sqrt{2+\sqrt{5}}\approx2.05817$. For convenience, if a mixed graph $M'$ is switching isomorphic to an induced subgraph of $M$, then we say $M'$ is an induced subgraph of $M$, or that $M$ contains $M'$; if a mixed graph $M'$ is not a subgraph of $M$, then we say $M$ is $M'$-\textit{free}. 
\begin{figure}
\begin{center}
\psfrag{1}{$C_4$}\psfrag{2}{$C_4'$}\psfrag{3}{$C_4''$}
\includegraphics[width=100mm]{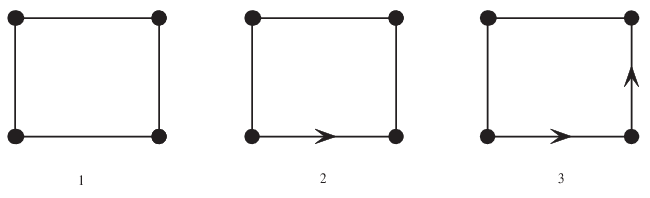} \\
  \caption{Mixed graphs $C_4,\,C_4'$ and $C_4''$.}\label{fig0}
\end{center}
\end{figure}

In order to present our main result, we need the definitions for some classes of mixed graphs. Let $T_{a,b,c}$ be the tree of order $a+b+c+1$ consisting of three paths of orders $a+1,\,b+1$ and $c+1$, respectively, where these paths have one end vertex in common. Let $Q_{a,b,c}$ be the tree of order $a+b+c+3$ consisting of a path $v_1v_2\cdots v_{a+b+c+1}$ with two extra edges affixed at $v_{a+1}$ and $v_{a+b+1}$, respectively.

We denote by $C_n,\, C_n'$ and $C_n''$ the $n$-vertex mixed cycles having no arc, just one arc and just two consecutive arcs with the same direction, respectively; mixed graphs $C_4,\,C_4'$ and $C_4''$ are depicted in Fig.~\ref{fig0}. Then by Lemma~\ref{lem2.4}, all positive, imaginary and negative mixed cycles of order $n$ are switching isomorphic to $C_n,\, C_n'$ and $C_n''$, respectively.

Given two positive integers $n>k\geq3,$ denote by $C_{k,n},\, C_{k,n}'$ and $C_{k,n}''$ the mixed graphs obtained by affixing a path of order $n-k+1$ at a vertex in $C_k,\, C_k'$ and $C_k''$, respectively.

Let $C_n,\, C_n'$ and $C_n''$ be mixed cycles with underlying graph $v_1v_2\cdots v_nv_1,$ then for $2\leq k\leq n,$ the mixed graphs $G_{n,k},\,G_{n,k}'$ and $G_{n,k}''$ are obtained from $C_n,\, C_n'$ and $C_n''$, respectively, by affixing an edge at $v_1$ and an edge at $v_k$; for $1\leq k\leq n-7,$ the mixed graph $U''_{k,n-6-k}$ is obtained from $C_6''$ by affixing a path of order $k+1$ at $v_1$ and a path of order $n-5-k$ at $v_4$; the mixed graph $U''_6$ is obtained from $C_6''$ by affixing a path of order $3$ at $v_1$, an edge at $v_3$ and an edge at $v_5$; the mixed graph $U''_8$ is obtained from $C_8''$ by affixing a path of order $3$ at $v_1$ and a path of order $3$ at $v_5$. 

Given positive integers $c\geq a,$ define
$$
b^\ast(a,c)=\left\{
              \begin{array}{ll}
                a+c+2, & \text{for $a\geq3;$} \\
                c+3, & \text{for $a=2;$} \\
                c, & \text{for $a=1.$}
              \end{array}
            \right.
$$
Then the main result of this section is:
\begin{figure}
\begin{center}
\psfrag{2}{$M^\ast; 2$}
\includegraphics[width=40mm]{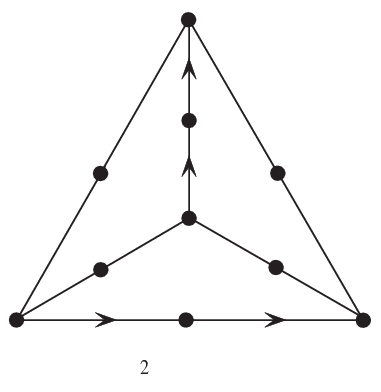} \\
  \caption{Mixed graph $M^\ast$ together with its spectral radius.}\label{fig01}
\end{center}
\end{figure}
\begin{thm}\label{thm3.5}
Let $M$ be a connected $C_4''$-free mixed graph of order $n$. Then $\rho(M)\leq \sqrt{2+\sqrt{5}}$ if and only if $M$ is an induced subgraph of one of the following mixed graphs:
\begin{wst}
\item[{\rm (i)}]$K_{1,4}$, $Q_{a,b,c}$ for $(a,b,c)\in\{(1,1,2),(2,4,2),(2,5,3),(3,7,3),(3,8,4)\};$ or $c\geq a>0$ and $b\geq b^\ast(a,c)$;
\item[{\rm (ii)}]$C_n$ for $n\geq3;$
\item[{\rm (iii)}]$C_n''$ for odd $n\geq3,$ $G''_{n-2,\frac{n}{2}}$ for even $n\geq10$, $U''_{k,n-6-k}$ for $1\leq k\leq n-7,$ $G''_{8,4},\,G''_{10,5},\,U''_6,\,U''_8;$
\item[{\rm (iv)}]$C'_{n-1,n}$ for $n\geq6,$ $C'_{3,n}$ for $n\geq4,$ $C'_{4,6}$.
\item[{\rm (v)}]$M^\ast,$ where $M^\ast$ is depicted in Fig.~\ref{fig01}.
\end{wst}
\end{thm}

In order to prove this result, we will distinguish the mixed graphs according to the number of cycles.
\subsection{\normalsize Mixed trees}
Smith \cite{S1970} characterized all simple graphs whose spectral radius does not exceed $2;$ Brouwer, Neumaier \cite{BN1989} and Cvetkovi\'c, Doob, Gutman \cite{CDG1989} determined all simple graphs with spectral radius between $2$ and $\rho^\ast.$
\begin{prop}[\cite{BN1989,CDG1989,S1970}]\label{prop3.1}
Let $G$ be a connected simple graph of order $n$. Then $\rho(G)\leq \sqrt{2+\sqrt{5}}$ if and only if $G$ is isomorphic to one of the graphs in $\{C_n, K_{1,4},P_n\}\cup \{T_{a,b,c}| a=1,\,b\geq c\geq1;$ or $a=b=2,\, c\geq2;$ or $a=2,\,b=c=3\}\cup\big\{Q_{a,b,c}\big|(a,b,c)\in\{(1,1,2),(2,4,2),(2,5,3),(3,7,3),(3,8,4)\};$ or $c\geq a>0,\,b\geq b^\ast(a,c)\big\}$.
\end{prop}
It follows from Lemma~\ref{lem2.5} that Proposition~\ref{prop3.1} provides all mixed trees (up to switching isomorphism) whose spectral radius is at most $\rho^\ast$. Each of these is an induced subgraph of one of the mixed graphs in (i) and (iii) of Theorem~\ref{thm3.5}.

\subsection{\normalsize Unicyclic mixed graphs}
A mixed graph is called \textit{unicyclic} if it is connected and contains exactly one mixed cycle.
\begin{lem}\label{lem3.01}
Let $n\geq4.$ Then $\rho(C'_{n-1,n})< \sqrt{2+\sqrt{5}}.$
\end{lem}
\begin{proof}
Let $u$ be the vertex of degree $1$ in $C'_{n-1,n}$, and let $v$ be the neighbour of $u.$ By Lemma~\ref{lem2.3} and Corollary~\ref{cor2.03}, one has
\begin{align}\notag
\Phi(C'_{n-1,n},x)&=x\Phi(C'_{n-1,n}-u,x)-\Phi(C'_{n-1,n}-u-v,x)\\ \notag
&=x[x\Phi(P_{n-2},x)-2\Phi(P_{n-3},x)]-\Phi(P_{n-2},x).
\end{align}
Together with \eqref{eq:2.3}, one has
\begin{align}\notag
\Phi(C'_{n-1,n},x)&=(x^2-1)\frac{\varphi(x)^{2n-2}-1}{\varphi(x)^{n}-\varphi(x)^{n-2}}
-2x\frac{\varphi(x)^{2n-4}-1}{\varphi(x)^{n-1}-\varphi(x)^{n-3}} \\ \label{eq:3.04}
&=\frac{1}{\varphi(x)^{n-2}(\varphi(x)^{2}-1)}[(x^2-1)(\varphi(x)^{2n-2}-1)
-2x\varphi(x)(\varphi(x)^{2n-4}-1)].
\end{align}
{Note that since $\rho^\ast=\sqrt{2+\sqrt{5}},$ we have that $\rho^\ast\varphi(\rho^\ast)=\frac{3+\sqrt{5}}{2}$ and $\varphi(\rho^\ast)^2=\frac{\sqrt{5}+1}{2}.$ 
Hence,}
\begin{align}\notag
\Phi(C'_{n-1,n},\rho^\ast)&=\frac{2}{(\sqrt{5}-1)\varphi(\rho^\ast)^{n-2}}[(1+\sqrt{5})
((\text{{\scriptsize $\frac{\sqrt{5}+1}{2}$}})^{n-1}-1)
-(3+\sqrt{5})((\text{{\scriptsize $\frac{\sqrt{5}+1}{2}$}})^{n-2}-1)] \\ \notag
&=\frac{4}{(\sqrt{5}-1)\varphi(\rho^\ast)^{n-2}} \\ \notag
&>0,
\end{align}
and so $\rho^\ast>\lambda_1(C'_{n-1,n})$ or $\rho^\ast<\lambda_2(C'_{n-1,n})$. On the other hand, by Lemmas~\ref{lem2.04} and \ref{lem2.1}, $\lambda_2(C'_{n-1,n})\leq\lambda_1(C'_{n-1,n}-u)<\rho^\ast,$ and hence $\rho^\ast>\lambda_1(C'_{n-1,n}).$ Finally, by Lemma \ref{lem2.02}, $\lambda_1(C'_{n-1,n})=\rho(C'_{n-1,n}),$ which finishes the proof.
\end{proof}
\begin{lem}\label{lem3.02}
Let $n\geq4.$ Then $\rho(C'_{3,n})\leq \sqrt{2+\sqrt{5}}.$
\end{lem}
\begin{proof}
{By Lemma \ref{lem2.1}, $\rho(C'_{3,n})$ is monotonically increasing, and by Lemma \ref{lem2.04}, $\rho(C'_{3,n})$ is bounded. Hence, $\lim_{m\rightarrow\infty}\rho (C'_{3,m})$ exists. On the other hand, by a direct computation, $\rho (C'_{3,6})=2.0285>2,$ and so by Lemma \ref{lem2.1}, $\lim_{m\rightarrow\infty}\rho(C'_{3,m})\geq\rho (C'_{3,6})> 2.$ Note that $C'_{3,n}$ contains no real mixed odd cycle, by Lemma \ref{lem2.02}, $\lim_{m\rightarrow\infty}\rho (C'_{3,m})=\lim_{m\rightarrow\infty}\lambda_1(C'_{3,m})>2.$ Then by Lemma \ref{lem2.07}, the value $\lim_{m\rightarrow\infty}\rho (C'_{3,m})$ is the largest positive root of
$$
\varphi(x)\Phi(C_3', x)-\Phi(P_2, x)=\frac{x+\sqrt{x^2-4}}{2}(x^3-3x)-(x^2-1),
$$ 
which is $\rho^\ast.$ By Lemma \ref{lem2.1}, for $n\geq4,\,\rho(C'_{3,n})\leq \lim_{m\rightarrow\infty}\rho (C'_{3,m})=\rho^\ast,$ as desired.}
\end{proof}
\begin{prop}\label{prop3.3}
Let $M$ be a connected $C_4''$-free unicyclic mixed graph. Then $\rho(M)\leq \sqrt{2+\sqrt{5}}$ if and only if $M$ is an induced subgraph of one of the following mixed graphs:
\begin{wst}
\item[{\rm (i)}]$C_n$ for $n\geq3;$
\item[{\rm (ii)}]$C_n''$ for odd $n\geq3,$ $G''_{n-2,\frac{n}{2}}$ for even $n\geq10$, $U''_{k,n-6-k}$ for $1\leq k\leq n-7,$ $G''_{8,4},\,G''_{10,5},\,U''_6,\,U''_8;$
\item[{\rm (iii)}]$C'_{n-1,n}$ for $n\geq6,$ $C'_{3,n}$ for $n\geq4,$ $C'_{4,6}$.
\end{wst}
\end{prop}
\begin{proof}
We note that a direct computation shows that $\rho(C'_{4,6})<\rho^\ast.$ Together with Lemmas \ref{lem3.01} and \ref{lem3.02}, this implies that all mixed graphs in {\rm (iii)} have spectral radii at most $\rho^\ast.$ Let $M$ be a connected $C_4''$-free unicyclic mixed graph with $\rho(M)\leq \rho^\ast.$ Let $\widetilde{C}$ be the unique mixed cycle (of length $k$) contained in $M.$ If $\widetilde{C}$ is real, then by Lemma~\ref{lem2.06}, $M$ is cycle-isomorphic to a signed graph. A result of Wang et al. \cite[Lemma 3.7]{WDHL2023} implies that $\rho(M)\leq \rho^\ast$ if and only if $M$ is an induced subgraph of 
one of the mixed graphs in {\rm (i)} and {\rm (ii)}.

In the following, we will consider the remaining case that $\widetilde{C}$ is imaginary. In this case, by Lemma~\ref{lem2.02}, the spectrum of $M$ is symmetric about $0,$ and so $\rho(M)=\lambda_1(M).$
\begin{figure}
\begin{center}
\psfrag{1}{$Z_1; 2.1358$}\psfrag{2}{$Z_2; 2.1753$}\psfrag{3}{$Z_3; 2.0743$}\psfrag{4}{$Z_4; 2.1358$}
\includegraphics[width=150mm]{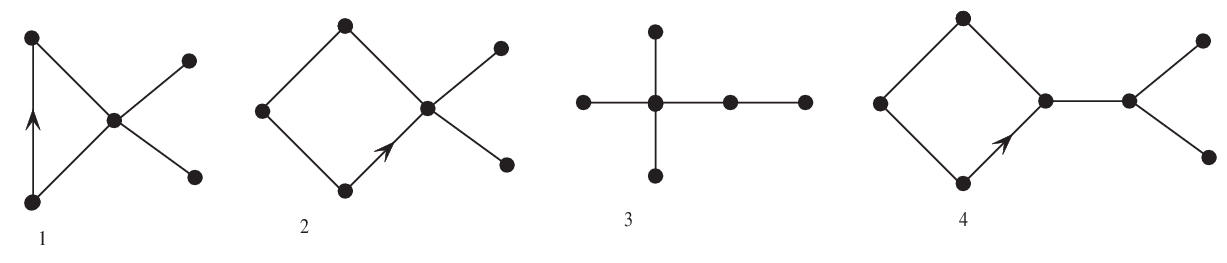} \\
  \caption{Mixed graphs $Z_1,\,Z_2,\,Z_3$ and $Z_4$ together with their spectral radii.}\label{fig1}
\end{center}
\end{figure}
\begin{claim}\label{claim1}
For all $v\in V(\widetilde{C}),$ one has $d_M(v)\leq3.$
\end{claim}
\begin{proof}
Suppose there is a vertex $v\in V(\widetilde{C})$ such that $d_M(v)\geq4.$ Depending on whether $k=3$, $k=4$, or $k\geq5,$ it follows that $Z_1,\,Z_2,$ or $Z_3,$ respectively, is an induced subgraph of $M,$ where $Z_1,\,Z_2$ and $Z_3$ are depicted in Fig.~\ref{fig1}. Direct computations give $\rho(Z_1)=2.1358>\rho^\ast,\,\rho(Z_2)=2.1753>\rho^\ast,\,\rho(Z_3)=2.0743>\rho^\ast.$ Thus, by Lemma~\ref{lem2.1}, $\rho(M)>\rho^\ast,$ a contradiction.
\end{proof}
\begin{claim}\label{claim2}
In $M$ there is at most one vertex in $V(\widetilde{C})$ with degree $3$.
\end{claim}
\begin{proof}
Suppose there are two vertices (say $v_1,v_i$) in $V(\widetilde{C})$ with degree $3.$ Then $M$ contains $M',$ an imaginary mixed $k$-cycle with $2$ pendent edges at $v_1$ and $v_i.$ Let $u_1,u_2$ be the pendent vertices in $M'$ with $u_1\sim v_1$ and $u_2\sim v_i,$ respectively. As $M$ is unicyclic, $M'$ is an induced subgraph of $M$.

By Lemma~\ref{lem2.3} and Corollary~\ref{cor2.03},
\begin{align}\notag
\Phi(M',x)&=x\Phi(M'-u_1,x)-\Phi(M'-u_1-v_1,x) \\ \notag
&=x[x\Phi(M'-u_1-u_2,x)-\Phi(M'-u_1-u_2-v_i,x)] \\ \notag
&\ \ \ -[x\Phi(M'-u_1-v_1-u_2,x)-\Phi(M'-u_1-v_1-u_2-v_i,x)] \\ \label{eq:3.1}
&=x(x^2-2)\Phi(P_{k-1},x)-2x^2\Phi(P_{k-2},x)+\Phi(P_{i-2},x)\cdot\Phi(P_{k-i},x).
\end{align}
Together with \eqref{eq:2.3} and \eqref{eq:3.1}, one has
\begin{align}\notag
\Phi(M',x)&=x(x^2-2)\cdot\frac{\varphi(x)^{2k}-1}{\varphi(x)^{k+1}-\varphi(x)^{k-1}}
-2x^2\cdot\frac{\varphi(x)^{2k-2}-1}{\varphi(x)^{k}-\varphi(x)^{k-2}} \\ \notag
&\ \ \ +\frac{\varphi(x)^{2i-2}-1}{\varphi(x)^{i}-\varphi(x)^{i-2}}\cdot\frac{\varphi(x)^{2(k-i)+2}-1}
{\varphi(x)^{k-i+2}-\varphi(x)^{k-i}}. 
\end{align}
Then 
\begin{align}\notag
\Phi(M',\rho^\ast)&=\frac{2}{(3-\sqrt{5})\varphi(\rho^\ast)^k}\cdot[\text{{\scriptsize $\frac{5+3\sqrt{5}}{2}$}}((\text{{\scriptsize $\frac{\sqrt{5}+1}{2}$}})^{k+1}
-(\text{{\scriptsize $\frac{\sqrt{5}+1}{2}$}})^{k}-\text{{\scriptsize $\frac{\sqrt{5}-1}{2}$}})\\ \notag
&\ \ \ -(4+2\sqrt{5})((\text{{\scriptsize $\frac{\sqrt{5}+1}{2}$}})^{k+1}
-(\text{{\scriptsize $\frac{\sqrt{5}+1}{2}$}})^{k}-1) \\ \label{eq:3.4}
&\ \ \ +(\text{{\scriptsize $\frac{\sqrt{5}+1}{2}$}})^{k+1}
-(\text{{\scriptsize $\frac{\sqrt{5}+1}{2}$}})^i
-(\text{{\scriptsize $\frac{\sqrt{5}+1}{2}$}})^{k-i+2}
+\text{{\scriptsize $\frac{\sqrt{5}+1}{2}$}}].
\end{align}
Note that
$$
(\text{{\scriptsize $\frac{\sqrt{5}+1}{2}$}})^i
+(\text{{\scriptsize $\frac{\sqrt{5}+1}{2}$}})^{k-i+2}
\geq 2(\text{{\scriptsize $\frac{\sqrt{5}+1}{2}$}})^{\frac{k+2}{2}}.
$$
Using this and \eqref{eq:3.4} gives 
\begin{align}\notag
\Phi(M',\rho^\ast)
\leq\frac{-2(\sqrt{5}+1)}{(3-\sqrt{5})\varphi(\rho^\ast)^k}
\cdot[(\text{{\scriptsize $\frac{\sqrt{5}+1}{2}$}})^{\frac{k}{2}}-2] <0
\end{align}
for $k\geq 3.$
Therefore, $\rho^\ast<\lambda_1(M')=\rho(M')\leq\rho(M),$ a contradiction.
\end{proof}
From Claims~\ref{claim1} and \ref{claim2}, we know that there is at most one vertex in $V(M)\backslash V(\widetilde{C})$ adjacent to some vertex in $V(\widetilde{C}).$ If $V(M)\backslash V(\widetilde{C})=\emptyset,$ then $M$ is switching isomorphic to $C'_{k}$.

If $|V(M)\backslash V(\widetilde{C})|=1,$ then $M$ is switching isomorphic to $C'_{k,k+1}$.

If $|V(M)\backslash V(\widetilde{C})|=2,$ then $M$ is switching isomorphic to $C'_{k,k+2}$. By Lemma~\ref{lem2.3},
\begin{align}\notag
\Phi(M,x)=(x^2-1)\Phi(\widetilde{C},x)-x\Phi(P_{k-1},x).
\end{align}
By using Corollary~\ref{cor2.03} and \eqref{eq:2.3}, one has
\begin{align}\notag
\Phi(M,x)&=x(x^2-2)\Phi(P_{k-1},x)-2(x^2-1)\Phi(P_{k-2},x) \\ \notag
&=x(x^2-2)\frac{\varphi(x)^{2k}-1}{\varphi(x)^{k+1}-\varphi(x)^{k-1}}
-2(x^2-1)\frac{\varphi(x)^{2k-2}-1}{\varphi(x)^{k}-\varphi(x)^{k-2}}.
\end{align}
Then
\begin{align}\notag
\Phi(M,\rho^\ast)
=\frac{2}{(\sqrt{5}-1)\varphi(\rho^\ast)^{k}}
[-(\text{{\scriptsize $\frac{\sqrt{5}+1}{2}$}})^{k-1}+\text{{\scriptsize $\frac{\sqrt{5}+7}{2}$}}].
\end{align}
Hence for $k\geq5,$ $\Phi(M,\rho^\ast)<0,$ and so $\rho^\ast<\lambda_1(M)=\rho(M),$ a contradiction. 

Because $\rho(C'_{k,k+2})>\rho^\ast$ for $k\geq5,$ we only need to consider the case $k<5$ when $|V(M)\backslash V(\widetilde{C})|\geq3.$ For $k=4,$ consider the mixed graphs $C'_{4,7}$ and $Z_4$ (see Fig.~\ref{fig1}). By a direct computation, $\rho(C'_{4,7})=2.0743>\rho^\ast$ and $\rho(Z_4)=2.1358>\rho^\ast,$ a contradiction. 

\begin{figure}
\begin{center}
\psfrag{2}{$T_s^\ast$}\psfrag{3}{$v_2$}\psfrag{4}{$v_3$}\psfrag{5}{$v_1$}\psfrag{6}{$u_s$}
\psfrag{7}{$u_2$}\psfrag{a}{$u_1$}\psfrag{8}{$w_1$}\psfrag{9}{$w_2$}
\psfrag{1}{$T_s^{\ast\ast}$}
\includegraphics[width=130mm]{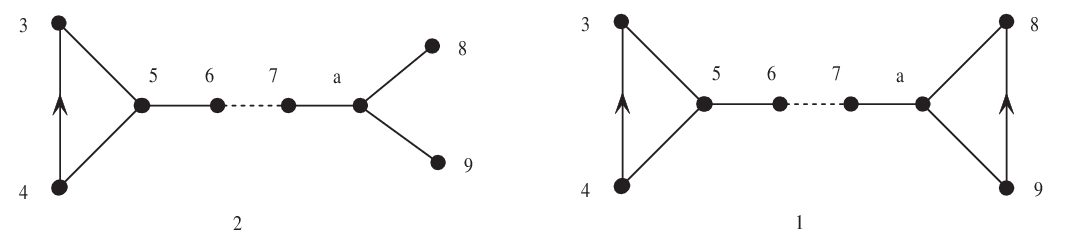} \\
  \caption{Mixed graphs $T_s^\ast$ and $T_s^{\ast\ast}$.}\label{fig2}
\end{center}
\end{figure}
For $k=3,$ if $M$ is not switching isomorphic to $C'_{3,n}$ for some $n\geq 6,$ then $M$ contains $T_s^\ast$ for some $s\geq 1,$ see Fig.~\ref{fig2}. By Lemma~\ref{lem2.3},
\begin{align}\notag
\Phi(T_s^\ast,x)&=x\Phi(T_s^\ast-v_2,x)-\Phi(T_s^\ast-v_2-v_1,x)-\Phi(T_s^\ast-v_2-v_3,x) \\ \notag
&=x[x\Phi(T_s^\ast-v_2-w_1,x)-\Phi(T_s^\ast-v_2-w_1-u_1,x)] \\ \notag
&\ \ \ -[x\Phi(T_s^\ast-v_2-v_1-w_1,x)-\Phi(T_s^\ast-v_2-v_1-w_1-u_1,x)] \\ \notag
&\ \ \ -[x\Phi(T_s^\ast-v_2-v_3-w_1,x)-\Phi(T_s^\ast-v_2-v_3-w_1-u_1,x)] \\ \notag
&=x^2\Phi(P_{s+3},x)-x\Phi(P_{s+2},x)-2x^2\Phi(P_{s+1},x)+x\Phi(P_s,x)+x^2\Phi(P_{s-1},x).
\end{align}
Using \eqref{eq:2.3}, one has
\begin{align}\notag
\Phi(T_s^\ast,x)&=x^2\frac{\varphi(x)^{2s+8}-1}{\varphi(x)^{s+5}-\varphi(x)^{s+3}}
-x\frac{\varphi(x)^{2s+6}-1}{\varphi(x)^{s+4}-\varphi(x)^{s+2}}
-2x^2\frac{\varphi(x)^{2s+4}-1}{\varphi(x)^{s+3}-\varphi(x)^{s+1}} \\ \notag
&\ \ \ +x\frac{\varphi(x)^{2s+2}-1}{\varphi(x)^{s+2}-\varphi(x)^{s}}
+x^2\frac{\varphi(x)^{2s}-1}{\varphi(x)^{s+1}-\varphi(x)^{s-1}}.
\end{align}
Then it follows that
\begin{align}\notag
\Phi(T_s^\ast,\rho^\ast)
=\frac{-(1+\sqrt{5})^2}{2\varphi(\rho^\ast)^{s+3}}<0.
\end{align}
Thus, $\rho^\ast<\lambda_1(T_s^\ast)\leq\lambda_1(M)=\rho(M),$ a contradiction. Therefore, 
$M$ is an induced subgraph of $C'_{3,n}$ for $n\geq6,$ which completes the proof.
\end{proof}
\begin{figure}
\begin{center}
\psfrag{1}{$\Theta_1; 2.5616$}\psfrag{2}{$\Theta_2; 2.3429$}\psfrag{3}{$\Theta_3; 2.2361$}\psfrag{4}{$\Theta_4; 2.3429$}\psfrag{5}{$\Theta_5; 2.5616$}
\includegraphics[width=140mm]{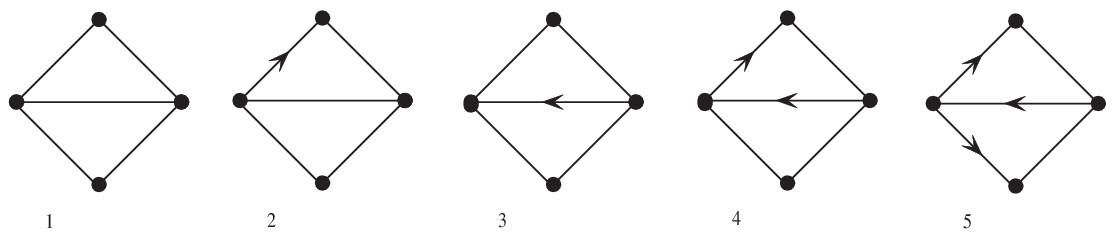} \\
  \caption{Mixed graphs $\Theta_1,\ldots,\Theta_5$ together with their spectral radii.}\label{fig3}
\end{center}
\end{figure}

\subsection{\normalsize Mixed graphs with at least two mixed cycles}
\begin{prop}\label{prop3.4}
If $M$ is a connected $C_4''$-free mixed graph with at least two mixed cycles, then $\rho(M)>\sqrt{2+\sqrt{5}},$ unless $M$ is an induced subgraph of $M^\ast,$ see Fig.~\ref{fig01}.
\end{prop}
\begin{proof}
Let $M$ be a connected $C_4''$-free mixed graph with at least two mixed cycles. By a direct computation, $\rho(M^\ast)=2<\rho^\ast.$ We only need to show that $M$ contains a mixed graph $M'$ with $\rho(M')>\rho^\ast$ when $M$ is not an induced subgraph of $M^\ast$. We consider $6$ cases depending on the girth of $M.$

{\bf Girth $3$}. Let $\widetilde{C}_3$ be a mixed $3$-cycle contained in $M.$ If $\widetilde{C}_3$ is real and there is a vertex in $V(M)\backslash V(\widetilde{C}_3)$ adjacent to exactly one vertex in $V(\widetilde{C}_3)$, then $M$ contains $C_{3,4}$ or $C''_{3,4}$. By Proposition~\ref{prop3.3}, $\rho(C_{3,4})>\rho^\ast$ and $\rho(C''_{3,4})>\rho^\ast,$ as desired. In the following, if there is a vertex in $V(M)\backslash V(\widetilde{C}_3)$ adjacent to exactly one vertex in $V(\widetilde{C}_3)$, we only need to consider the case that $\widetilde{C}_3$ is imaginary. We now distinguish cases 3.1, 3.2 and 3.3. 

{\bf 3.1)}. Suppose there is an induced mixed $k$-cycle sharing one edge with $\widetilde{C}_3$.

{For $k=3$, let $v$ be a vertex in $V(M)\backslash V(\widetilde{C}_3)$ adjacent to two or three vertices in $V(\widetilde{C}_3)$. Let $\widetilde{C}_3$ be positive (resp. negative, imaginary) and fix it. Since $M$ is $C_4''$-free, consider all the possibilities of the edges between $v$ and $V(\widetilde{C}_3)$ such that $M[V(\widetilde{C}_3)\cup \{v\}]$ is $C_4''$-free, count the numbers of positive (resp. negative, imaginary) triangles and positive (resp. imaginary) quadrangles contained in $M[V(\widetilde{C}_3)\cup \{v\}]$. By using Lemma~\ref{lem2.4}, it follows that $M[V(\widetilde{C}_3)\cup \{v\}]$ is switching isomorphic to one of $\Theta_i,\, i=1,2,\ldots,9,$ see Fig.~\ref{fig3} and Fig.~\ref{fig4}. By direct computations, all of these mixed graphs have spectral radii larger than $\rho^\ast,$ as desired. }

\begin{figure}
\begin{center}
\psfrag{1}{$\Theta_6; 3$}\psfrag{2}{$\Theta_7; 2.7093$}\psfrag{3}{$\Theta_8; 2.7093$}\psfrag{4}{$\Theta_9; 3$}
\includegraphics[width=120mm]{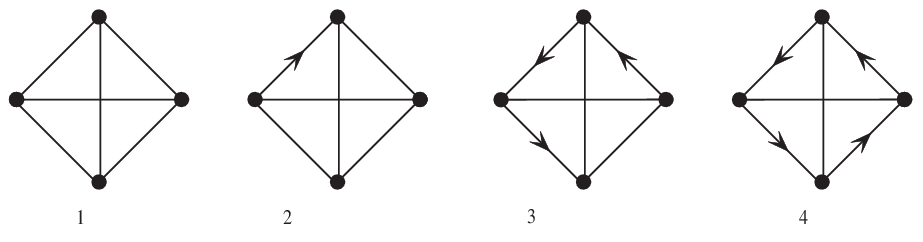} \\
  \caption{Mixed graphs $\Theta_6,\ldots,\Theta_9$ together with their spectral radii.}\label{fig4}
\end{center}
\end{figure}

\begin{figure}
\begin{center}
\psfrag{1}{$\Theta_{10}; 2.2882$}\psfrag{2}{$\Theta_{11}; 2.2303$}\psfrag{3}{$\Theta_{12}; 2.2303$}\psfrag{4}{$Z_5; 2.2361$}
\includegraphics[width=140mm]{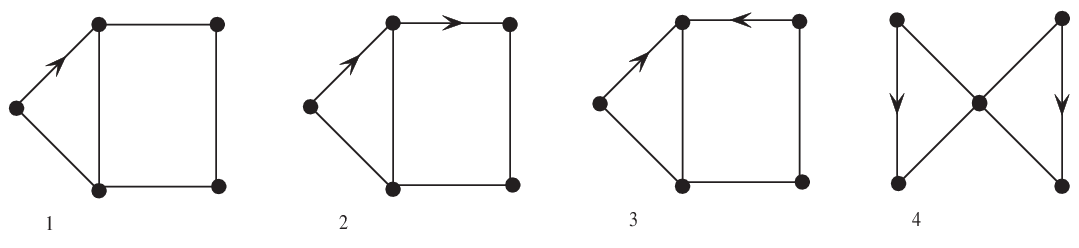} \\
  \caption{Mixed graphs $\Theta_{10},\Theta_{11},\Theta_{12}$ and $Z_5$ together with their spectral radii.}\label{fig5}
\end{center}
\end{figure}

For $k=4,$ we may assume there is no mixed $3$-cycle sharing one edge with $\widetilde{C}_3.$ Since $M$ is $C_4''$-free, $M$ contains one of $\Theta_{10},\Theta_{11}$ and $\Theta_{12}$, see Fig.~\ref{fig5}. By direct computations, all of these mixed graphs have spectral radii larger than $\rho^\ast$, as desired.

For $k\geq5,$ we may assume there is no mixed $3$-cycle and no mixed $4$-cycle sharing one edge with $\widetilde{C}_3.$ Then $M$ contains $G'_{3,2}$. By Proposition~\ref{prop3.3}, $\rho(G'_{3,2})>\rho^\ast,$ as desired.

{\bf 3.2)}. Suppose there is no induced mixed cycle sharing an edge with $\widetilde{C}_3$, but an induced mixed $k$-cycle sharing one vertex with $\widetilde{C}_3$.

For $k=3,$ $M$ contains $Z_5,$ see Fig.~\ref{fig5}. By a direct computation, $\rho(Z_5)=2.2361>\rho^\ast$, as desired.
For $k\geq4,$ $M$ contains $Z_1$, see Fig.~\ref{fig1}. By Proposition~\ref{prop3.3}, $\rho(Z_1)>\rho^\ast,$ as desired.

{\bf 3.3)}. Suppose there is no induced mixed cycle sharing a vertex with $\widetilde{C}_3$.

Now $M$ contains $T_{s}^\ast$ or $T_s^{\ast\ast},$ see Fig.~\ref{fig2}. By Proposition~\ref{prop3.3}, $\rho(T_{s}^\ast)>\rho^\ast,$ as desired.
For $T_s^{\ast\ast},$ by Lemma~\ref{lem2.3}, one has
\begin{align}\notag
\Phi(T_s^{\ast\ast},x)&=x\Phi(T_s^{\ast\ast}-w_1,x)-\Phi(T_s^{\ast\ast}-w_1-u_1,x)
-\Phi(T_s^{\ast\ast}-w_1-w_2,x) \\ \notag
&=x[x\Phi(T_s^{\ast\ast}-w_1-v_2,x)-\Phi(T_s^{\ast\ast}-w_1-v_2-v_1,x)-\Phi(T_s^{\ast\ast}
-w_1-v_2-v_3,x)] \\ \notag
&\ \ \ -[x\Phi(T_s^{\ast\ast}-w_1-u_1-v_2,x)-\Phi(T_s^{\ast\ast}-w_1-u_1-v_2-v_1,x) \\ \notag
&\ \ \ -\Phi(T_s^{\ast\ast}-w_1-u_1-v_2-v_3,x)]-[x\Phi(T_s^{\ast\ast}-w_1-w_2-v_2,x) \\ \notag
&\ \ \ -\Phi(T_s^{\ast\ast}-w_1-w_2-v_2-v_1,x)-\Phi(T_s^{\ast\ast}-w_1-w_2-v_2-v_3,x)] \\ \notag
&=x^2\Phi(P_{s+3},x)-2x\Phi(P_{s+2},x)-(2x^2-1)\Phi(P_{s+1},x)+2x\Phi(P_s,x)+x^2\Phi(P_{s-1},x).
\end{align}
Together with \eqref{eq:2.3},
\begin{align}\notag
\Phi(T_s^{\ast\ast},x)&=x^2\cdot\frac{\varphi(x)^{2s+8}-1}{\varphi(x)^{s+5}-\varphi(x)^{s+3}}
-2x\cdot\frac{\varphi(x)^{2s+6}-1}{\varphi(x)^{s+4}-\varphi(x)^{s+2}}
-(2x^2-1)\cdot\frac{\varphi(x)^{2s+4}-1}{\varphi(x)^{s+3}-\varphi(x)^{s+1}} \\ \notag
&\ \ \ +2x\cdot\frac{\varphi(x)^{2s+2}-1}{\varphi(x)^{s+2}-\varphi(x)^{s}}+
x^2\cdot\frac{\varphi(x)^{2s}-1}{\varphi(x)^{s+1}-\varphi(x)^{s-1}}.
\end{align}
Now it follows that
\begin{align}\notag
\Phi(T_s^{\ast\ast},\rho^\ast)
=\frac{-(\sqrt{5}+1)^2}{\varphi(\rho^\ast)^{s+3}}<0.
\end{align}
Thus, $\rho^\ast<\lambda_1(T_s^{\ast\ast})\leq\rho(T_s^{\ast\ast}),$ as desired.
\begin{figure}
\begin{center}
\psfrag{1}{$\Theta_{13}; 2.4495$}\psfrag{2}{$\Theta_{14}; 2.2882$}\psfrag{3}{$\Theta_{15}; 2.2562$}\psfrag{4}{$\Theta_{16}; 2.2562$}
\includegraphics[width=120mm]{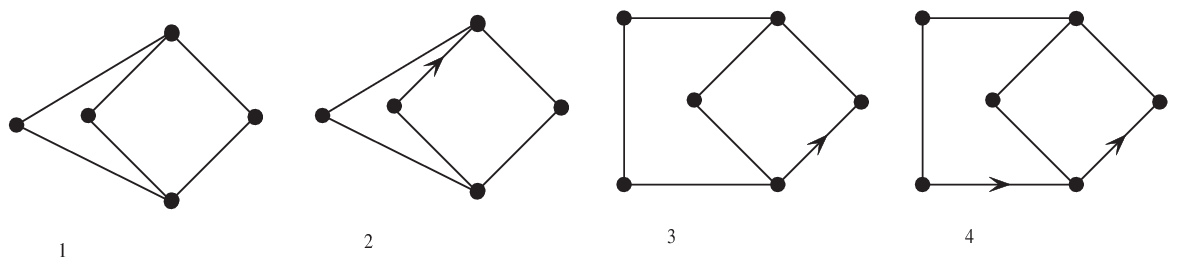} \\
  \caption{Mixed graphs $\Theta_{13},\ldots,\Theta_{16}$ together with their spectral radii.}\label{fig7}
\end{center}
\end{figure}

{\bf Girth 4}. Let $\widetilde{C}_4$ be a mixed $4$-cycle contained in $M.$ Since $M$ is $C_4''$-free, $\widetilde{C}_4$ is positive or imaginary.

If $\widetilde{C}_4$ is positive, then $M$ contains $C_{4,5},\,\Theta_{13},$ or $\Theta_{14},$ where $\Theta_{13}$ and $\Theta_{14}$ are depicted in Fig.~\ref{fig7}. By Proposition \ref{prop3.3} and direct computations, all of these mixed graphs have spectral radii larger than $\rho^\ast,$ as desired.

If $\widetilde{C}_4$ is imaginary, then $M$ contains $C'_{4,7},\,G'_{4,2},\,G'_{4,3},\,Z_2,\,Z_4,\,\Theta_{14},\, \Theta_{15},\,\Theta_{16},\,\Theta_{17},$ or $\Theta_{18},$ where $Z_2$ and $Z_4$ are depicted in Fig.~\ref{fig1}, $\Theta_{14},\,\Theta_{15},\,\Theta_{16}$ are depicted in Fig.~\ref{fig7}, $\Theta_{17}$ and $\Theta_{18}$ are depicted in Fig.~\ref{fig8}. By Proposition~\ref{prop3.3} and direct computations, all of these mixed graphs have spectral radii larger than $\rho^\ast,$ as desired.




\begin{figure}
\begin{center}
\psfrag{1}{$\Theta_{17}; 2.1701$}\psfrag{2}{$\Theta_{18}; 2.2361$}\psfrag{3}{$\Theta_{19}; 2.2361$}\psfrag{4}{$\Theta_{20}; 2.1358$}
\includegraphics[width=120mm]{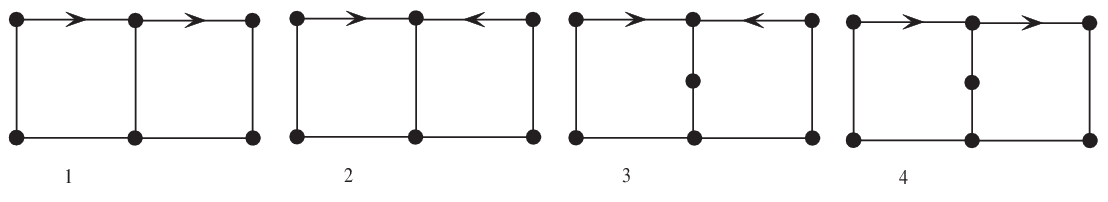} \\
  \caption{Mixed graphs $\Theta_{17},\ldots,\Theta_{20}$ together with their spectral radii.}\label{fig8}
\end{center}
\end{figure}






{\bf Girth 5}. Let $\widetilde{C}_5$ be a mixed $5$-cycle contained in $M.$ Depending on whether $\widetilde{C}_5$ shares two adjacent edges with another induced mixed $k$-cycle, $M$ contains one of $C_{5,6},\,C''_{5,6},\,\Theta_{19},\,\Theta_{20}\,(k=5)$ and $G_{5,3},\,G'_{5,3},\,G''_{5,3}\,(k\geq6),$ where $\Theta_{19}$ and $\Theta_{20}$ are depicted in Fig.~\ref{fig8}, or one of $C_{5,6},\,C''_{5,6}$ and $C'_{5,7}$. By Proposition~\ref{prop3.3} and direct computations, all of these mixed graphs have spectral radii larger than $\rho^\ast,$ as desired.





\begin{figure}
\begin{center}
\psfrag{1}{$\Theta_{21}; 2.1149$}\psfrag{2}{$\Theta_{k}'$}\psfrag{3}{$v_1$}\psfrag{4}{$v_2$}\psfrag{5}{$v_3$}
\psfrag{6}{$v_{k-4}$}\psfrag{7}{$v_{k-3}$}\psfrag{8}{$v_{k-2}$}
\psfrag{9}{$w_1$}\psfrag{a}{$w_2$}\psfrag{b}{$u_1$}\psfrag{c}{$u_2$}
\includegraphics[width=100mm]{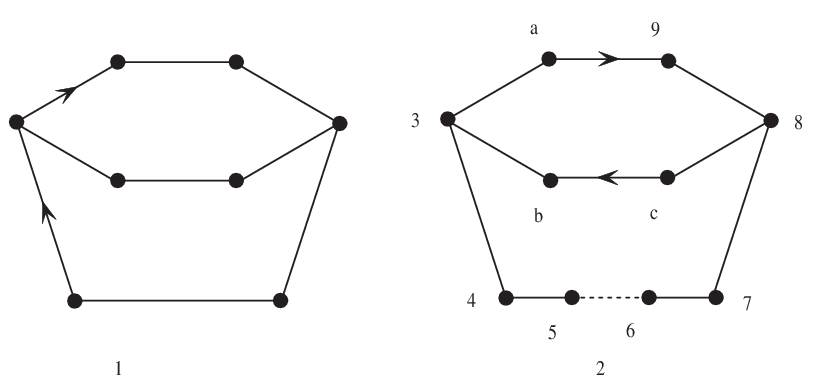} \\
  \caption{Mixed graphs $\Theta_{21}$ and $\Theta_{k}'$.}\label{fig9}
\end{center}
\end{figure}

{\bf Girth 6}. Let $\widetilde{C}_6$ be a mixed $6$-cycle contained in $M.$ We now distinguish cases 6.1, 6.2 and 6.3. 

{\bf 6.1)}. Suppose there is an induced mixed $k$-cycle sharing a path of length $3$ with $\widetilde{C}_6$.

For $k=6,$ $M$ contains one of $C_{6,7}$ and $\Theta_{21},$ where $\Theta_{21}$ is as depicted in Fig.~\ref{fig9}. By Proposition~\ref{prop3.3} and a direct computation, both $C_{6,7}$ and $\Theta_{21}$ have spectral radii larger than $\rho^\ast,$ as desired.

For $k\geq7$, $M$ contains one of $C_{6,7}$, $C'_{6,8},$ $C_{k,k+1}$ and $\Theta'_k,$ where $\Theta'_k$ is as depicted in Fig.~\ref{fig9}. By Proposition~\ref{prop3.3}, all of $C_{6,7}$, $C'_{6,8}$ and $C_{k,k+1}$ have spectral radii larger than $\rho^\ast$. For the mixed graph $\Theta'_k,$ by Lemma~\ref{lem2.3}, one has
\begin{align}\notag
\Phi(\Theta'_k,x)&=x\Phi(\Theta'_k-u_1,x)-\Phi(\Theta'_k-u_1-v_1,x)
-\Phi(\Theta'_k-u_1-u_2,x)+2\Phi(\Theta'_k-V(\widetilde{C}_6),x) \\ \notag
&=x[x\Phi(\Theta'_k-u_1-w_1,x)-\Phi(\Theta'_k-u_1-w_1-w_2,x)-\Phi(\Theta'_k-u_1-w_1-v_{k-2},x)] \\ \notag
&\ \ \ -[x\Phi(\Theta'_k-u_1-v_1-u_2,x)-\Phi(\Theta'_k-u_1-v_1-u_2-v_{k-2},x)] \\ \notag
&\ \ \ -\Phi(\Theta'_k-u_1-u_2,x)+2\Phi(\Theta'_k-V(\widetilde{C}_6),x) \\ \label{eq:3.8}
&=x^2\Phi(P_{k},x)-2x\Phi(P_{k-1},x)-x^2\Phi(P_{k-2},x)+(x^2+1)\Phi(P_{k-4},x)-\Phi(\widetilde{C}'_k,x).
\end{align}
By Corollary~\ref{cor2.03}, $\Phi(\widetilde{C}'_k,x)=x\Phi(P_{k-1},x)-2\Phi(P_{k-2},x).$ Together with \eqref{eq:2.3} and \eqref{eq:3.8}, one has
\begin{align}\notag
\Phi(\Theta'_k,x)&=x^2\cdot\frac{\varphi(x)^{2k+2}-1}{\varphi(x)^{k+2}-\varphi(x)^{k}}
-3x\cdot\frac{\varphi(x)^{2k}-1}{\varphi(x)^{k+1}-\varphi(x)^{k-1}} \\ \notag
&\ \ \ -(x^2-2)\frac{\varphi(x)^{2k-2}-1}{\varphi(x)^{k}-\varphi(x)^{k-2}}
+(x^2+1)\frac{\varphi(x)^{2k-6}-1}{\varphi(x)^{k-2}-\varphi(x)^{k-4}}. 
\end{align}
Now it follows that
\begin{align}\notag
\Phi(\Theta'_k,\rho^\ast)=\frac{-(1+\sqrt{5})^2}{\varphi(\rho^\ast)^{k}}<0.
\end{align}
Thus, $\rho^\ast<\lambda_1(\Theta'_k)\leq\rho(\Theta'_k),$ as desired.
\begin{figure}
\begin{center}
\psfrag{2}{$F^\ast$}\psfrag{1}{$t$}\psfrag{4}{$G_t^\ast$}
\includegraphics[width=100mm]{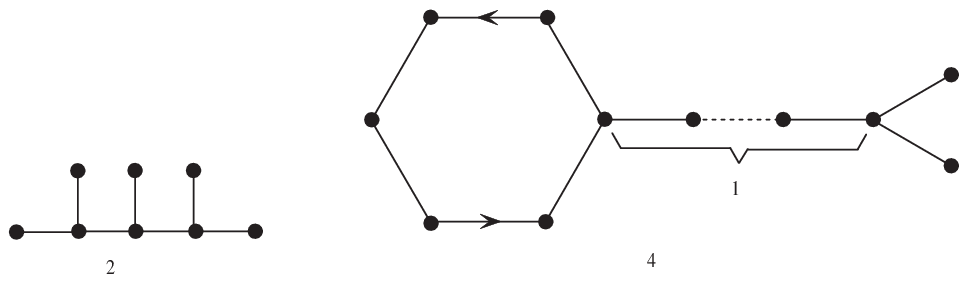} \\
  \caption{Mixed graphs $F^\ast$ and $G_t^\ast$.}\label{fig10}
\end{center}
\end{figure}

{\bf 6.2)}. Suppose there is no induced mixed cycle sharing a path of length $3$ with $\widetilde{C}_6$, but an induced mixed $k$-cycle sharing two adjacent edges with $\widetilde{C}_6$.

For $k=6,$ if $M$ is not an induced subgraph of $M^\ast$, then $M$ contains one of $C_{6,7}$, $C'_{6,8},$ $G''_{6,2},$ $C_{8,9}$ and $F^\ast$ (see Fig.~\ref{fig10}). By Propositions~\ref{prop3.1} and \ref{prop3.3}, all of these mixed graphs have spectral radii larger than $\rho^\ast,$ as desired.

For $k\geq7,$ $M$ contains one of $C_{k,k+1}$, $C'_{k,k+2}$ and $G''_{k,3}$. By Proposition~\ref{prop3.3}, all of these mixed graphs have spectral radii larger than $\rho^\ast,$ as desired.

{\bf 6.3)}. Suppose there is no induced mixed cycle sharing a path of length $2$ or $3$ with $\widetilde{C}_6$.

Now $M$ contains one of $C_{6,7}$, $C'_{6,8}$, $G''_{6,2}$ and $G^\ast_t$ (see Fig.~\ref{fig10}) for some $t\geq1$. By Proposition~\ref{prop3.3}, all of these mixed graphs have spectral radii larger than $\rho^\ast,$ as desired.



{\bf Girth 8}. Let $\widetilde{C}_8$ be a mixed $8$-cycle contained in $M.$ If $\widetilde{C}_8$ is positive or imaginary, then $M$ contains $C_{8,9}$ or $C'_{8,10}$. By Proposition~\ref{prop3.3}, both $C_{8,9}$ and $C'_{8,10}$ have spectral radii larger than $\rho^\ast,$ as desired. In the following, we consider the case that $\widetilde{C}_8$ is negative, and distinguish 2 cases.

{\bf 8.1)}. Suppose there is a mixed $k$-cycle sharing a path of length $4$ with $\widetilde{C}_8.$

For $k=8,$ $M$ contains $3$ mixed $8$-cycles. It is not possible that all of these mixed $8$-cycles are negative. Hence, $M$ contains $C_{8,9}$ or $C'_{8,10}$. By Proposition~\ref{prop3.3}, both $C_{8,9}$ and $C'_{8,10}$ have spectral radii larger than $\rho^\ast,$ as desired. For $k\geq9,$ $M$ contains $C''_{8,11}$. By Proposition~\ref{prop3.3}, $\rho(C''_{8,11})>\rho^\ast,$ as desired.

{\bf 8.2)}. Suppose there is no mixed cycle sharing a path of length $4$ with $\widetilde{C}_8.$

Now $M$ contains $C''_{8,11}$. By Proposition~\ref{prop3.3}, $\rho(C''_{8,11})>\rho^\ast,$ as desired.

{\bf Girth 7 or at least 9}. $M$ contains one of $C_{k,k+1},\,C'_{k,k+2}$ for $k=7$ or $k\geq 9$, $C''_{k,k+1}$ for $k=7$ and $C''_{k,k+3}$ for $k\geq 9$. By Proposition~\ref{prop3.3}, all of these mixed graphs have spectral radii larger than $\rho^\ast,$ as desired.
\end{proof}
\begin{proof}[Proof of Theorem~\ref{thm3.5}]
The proof follows by Propositions~\ref{prop3.1}, \ref{prop3.3} and \ref{prop3.4}.
\end{proof}

\section{\normalsize Limit points for the spectral radius of the Hermitian adjacency matrix of mixed graphs}\setcounter{equation}{0}
In this section, based on the results in \cite{H1972,S1989,BB2024} and Theorem \ref{thm3.5}, we are to establish all the $H$-limit points for mixed graphs, thereby completely solving the first part of the Hoffman program with respect to the Hermitian adjacency matrix of mixed graphs.

For every integer $k>0,$ let $\beta_k$ be the largest positive root of $x^k(x^2-x-1)+1.$ Define
\begin{align}\label{eq:4.01}
\eta_k=\beta_k^\frac{1}{2}+\beta_k^{-\frac{1}{2}}.
\end{align} 

Hoffman \cite{H1972} and Shearer \cite{S1989} determined all $A$-limit points for simple graphs.
\begin{prop}[\cite{H1972}]\label{prop4.1}
The numbers 
$
2=\eta_1<\eta_2<\cdots
$ 
are precisely the limit points for the adjacency spectral radii of simple graphs smaller than $\sqrt{2+\sqrt{5}}.$ Moreover, $\lim_{n\rightarrow\infty}\rho(T_{1,k,n})=\eta_k.$
\end{prop}
\begin{prop}[\cite{S1989}]\label{prop4.2}
Each real number $\lambda\geq\sqrt{2+\sqrt{5}}$ is a limit point for the adjacency spectral radii of a suitable sequence of trees.
\end{prop}

For every integer $k\geq0,$ let $\gamma_k$ be the largest positive root of $x^{k+2}(x^2-x-1)+x^2-1,$
and let $\vartheta$ be the largest positive root of $x^6-2x^5+x^4-x^2+x-1.$ Define
\begin{align}\label{eq:4.02}
\text{$\zeta_k=\gamma_k^\frac{1}{2}+\gamma_k^{-\frac{1}{2}}$ and $\xi=\vartheta^\frac{1}{2}+\vartheta^{-\frac{1}{2}}$.}
\end{align} 

Belardo and Brunetti \cite{BB2024} determined all $A$-limit points for signed graphs.
\begin{prop}[\cite{BB2024}]\label{prop4.3}
The numbers $\xi$ and $\zeta_0<\zeta_1<\cdots$ are precisely the limit points for the adjacency spectral radii of signed graphs which cannot be obtained from sequences of simple graphs.
\end{prop}
\begin{figure}
\begin{center}
\psfrag{1}{$k$}\psfrag{2}{$n$}\psfrag{3}{$\Omega_{k,n}$}\psfrag{4}{$\Omega'_n$}
\includegraphics[width=140mm]{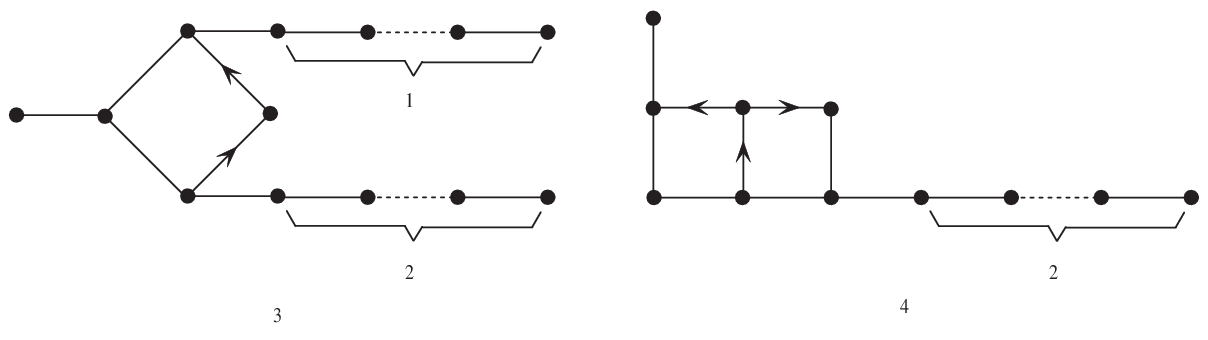} \\
  \caption{Mixed graphs $\Omega_{k,n}$ and $\Omega'_n$.}\label{fig11}
\end{center}
\end{figure}

\begin{prop}\label{prop4.4} 
For $k\geq0,$ let $\Omega_{k,n}$ and $\Omega'_n$ be the mixed graphs depicted in Fig.~\ref{fig11}. Then
$$
\text{$\lim_{n\rightarrow\infty}\rho(\Omega_{k,n})=\zeta_k$   and $\lim_{n\rightarrow\infty}\rho(\Omega_{n}')=\xi$.}  
$$
\end{prop}
\begin{proof}
The result follows from Lemma~\ref{lem2.06} and \cite[Proposition 4.2]{BB2024}.
\end{proof}
\begin{lem}\label{lem4.5} 
It holds that 
$$
\text{$\lim_{n\rightarrow\infty}\rho(C'_n)=2$ and $\lim_{n\rightarrow\infty}\rho(C'_{n-1,n})=\lim_{n\rightarrow\infty}\rho(C'_{3,n})=\sqrt{2+\sqrt{5}}.$}  
$$
\end{lem}
\begin{proof}
By Lemma \ref{lem2.1} and \cite[Theorem 5.7]{GM2017}, $\rho(P_{n-1})\leq\rho(C'_n)\leq\rho(C_n),$ and so
$$
2=\lim_{n\rightarrow\infty}\rho(P_{n-1})\leq\lim_{n\rightarrow\infty}\rho(C'_n)
\leq\lim_{n\rightarrow\infty}\rho(C_n)=2.
$$
Hence, $\lim_{n\rightarrow\infty}\rho(C'_n)=2.$

{By Theorem \ref{thm3.5}, we have that $\rho(C'_{n-1,n})\leq\rho^\ast.$ On the other hand, $C'_{n-1,n}$ contains $T_{1,\lfloor\frac{n-3}{2}\rfloor,\lfloor\frac{n-3}{2}\rfloor}$ as an induced subgraph. By Lemma \ref{lem2.1} and \cite[Proposition 3.6]{H1972}, one has
$$
\rho^\ast=\lim_{n\rightarrow\infty}\rho(T_{1,\lfloor\frac{n-3}{2}\rfloor,\lfloor\frac{n-3}{2}\rfloor})
\leq\lim_{n\rightarrow\infty}\rho(C'_{n-1,n})\leq\rho^\ast,
$$
which gives $\lim_{n\rightarrow\infty}\rho(C'_{n-1,n})=\rho^\ast.$

Finally, in the proof of Lemma \ref{lem3.02}, we have shown $\lim_{n\rightarrow\infty}\rho(C'_{3,n})=\rho^\ast.$ 
}



\end{proof}
\begin{thm}\label{thm4.6}
A real number $\zeta$ is a limit point for the spectral radii of mixed graphs if and only if $\zeta\in\{\eta_k|k>0\}\cup\{\xi\}\cup\{\zeta_k|k\geq0\}\cup[\sqrt{2+\sqrt{5}},+\infty),$
where $\eta_k\,(k>0)$ are defined in \eqref{eq:4.01}, $\xi$ and $\zeta_k\,(k\geq0)$ are defined in \eqref{eq:4.02}.
\end{thm}
\begin{proof}
One direction follows by Propositions \ref{prop4.1}, \ref{prop4.2} and \ref{prop4.4}. To show that there are no other limit points, let $\zeta$ be a limit point for the spectral radii of mixed graphs, and let $(M_k)_{k\in \mathbb{N}}$ be a sequence of connected mixed graphs such that
$$
\text{$\rho(M_i)\neq \rho(M_j)$ whenever $i\neq j,$ and $\lim_{k\rightarrow\infty}\rho(M_k)=\zeta.$}
$$ 

If there is an infinite subsequence $(M_{k_i})_{i\in \mathbb{N}}\subseteq(M_k)_{k\in \mathbb{N}}$ such that each mixed graph in $(M_{k_i})_{i\in \mathbb{N}}$ contains no imaginary mixed cycle, then it follows from Lemma \ref{lem2.06} and Propositions \ref{prop4.1}-\ref{prop4.3} that $\zeta$ is one of the stated limit points. In the following, assume there are finitely many mixed graphs in $(M_k)_{k\in \mathbb{N}}$ containing no imaginary mixed cycles.

If there is an infinite subsequence $(M_{k_i})_{i\in \mathbb{N}}\subseteq(M_k)_{k\in \mathbb{N}}$ such that each mixed graph in $(M_{k_i})_{i\in \mathbb{N}}$ contains an imaginary mixed triangle, then by Lemma \ref{lem2.05}, $\lim_{i\rightarrow\infty}\diam(\Gamma(M_{k_i}))=+\infty.$ Then for each $s\in \mathbb{N}$, there is $N\in \mathbb{N}$ such that $M_{k_i}$ contains $C'_{3,s}$ as an induced subgraph, and so $\rho(M_{k_i})\geq\rho(C'_{3,s})$ whenever $i\geq N.$ Hence by Lemma~\ref{lem4.5},
$$
\zeta=\lim_{k\rightarrow\infty}\rho(M_k)=\lim_{i\rightarrow\infty}\rho(M_{k_i})
\geq\lim_{s\rightarrow\infty}\rho(C'_{3,s})=\sqrt{2+\sqrt{5}}.
$$

If there is an infinite subsequence $(M_{k_i})_{i\in \mathbb{N}}\subseteq(M_k)_{k\in \mathbb{N}}$ such that each mixed graph in $(M_{k_i})_{i\in \mathbb{N}}$ contains an induced imaginary mixed $t$-cycle for some fixed $t\geq 4.$ Then by Proposition~\ref{prop3.3}, $\rho(M_{k_i})>\sqrt{2+\sqrt{5}}$ for all $M_{k_i}$ with $\diam(\Gamma(M_{k_i}))\geq\lceil\frac{t}{2}\rceil+3.$ From Lemma \ref{lem2.05}, it then follows that
$$
\zeta=\lim_{k\rightarrow\infty}\rho(M_k)=\lim_{i\rightarrow\infty}\rho(M_{k_i})>\sqrt{2+\sqrt{5}}.
$$

If for all fixed $t\geq 3,$ there are finitely many mixed graphs in $(M_k)_{k\in \mathbb{N}}$ containing an induced imaginary mixed $t$-cycle, then 
either there is an infinite subsequence $(M_{k_i})_{i\in \mathbb{N}}\subseteq(M_k)_{k\in \mathbb{N}}$ such that each mixed graph in $(M_{k_i})_{i\in \mathbb{N}}$ is an imaginary mixed cycle, or there is an infinite subsequence $(M_{k_i})_{i\in \mathbb{N}}\subseteq(M_k)_{k\in \mathbb{N}}$ such that for each $N\in \mathbb{N},$ there is a mixed graph in $(M_{k_i})_{i\in \mathbb{N}}$ containing $C'_{n-1,n}$ as an induced subgraph for some $n\geq N$. In the former case, by Lemma \ref{lem4.5},
$$
\zeta=\lim_{k\rightarrow\infty}\rho(M_k)=\lim_{i\rightarrow\infty}\rho(M_{k_i})=2.
$$
In the latter case, by Lemmas \ref{lem2.1} and \ref{lem4.5},
$$
\zeta=\lim_{k\rightarrow\infty}\rho(M_k)=\lim_{i\rightarrow\infty}\rho(M_{k_i})
\geq\lim_{n\rightarrow\infty}\rho(C'_{n-1,n})=\sqrt{2+\sqrt{5}}.
$$

Hence, if $\zeta$ is a limit point for the spectral radii of mixed graphs, then $\zeta$ must be one of the stated limit points.
\end{proof}

\end{document}